\documentclass[10pt,a4paper]{article}
\usepackage{geometry}
\usepackage[english]{babel}
\usepackage{amsmath}
\usepackage{graphicx}
\usepackage{ marvosym }
\usepackage[utf8]{inputenc}
\usepackage{mathtools}
\usepackage{geometry}
\usepackage{amsfonts}
\usepackage{amssymb}
\usepackage{amsthm}
\usepackage{t1enc}
\usepackage[titles]{tocloft}
\usepackage{makeidx}
\usepackage{wasysym}
\usepackage{stmaryrd}
\usepackage{calc}  
\usepackage{enumitem} 
\usepackage[refpage]{nomencl}

\usepackage[colorlinks=true,urlcolor=blue, linkcolor=blue,pageanchor=false, backref]{hyperref}
\usepackage{tikz}
\usetikzlibrary{arrows}
\usepackage{indentfirst}
\usepackage{xcolor}
\usepackage{ dsfont }
\usepackage{float}
\usepackage[abbrev,msc-links,backrefs]{amsrefs} 
\usepackage{doi}

\renewcommand{\PrintDOI}[1]{\doi{#1}}

\theoremstyle{plain}
\newtheorem{thm}{Theorem}[section]

\newtheorem{claim}[thm]{Claim}

\newtheorem{prop}[thm]{Proposition}

\newtheorem*{prop*}{Proposition}
\newtheorem*{seged*}{Sublemma}

\newtheorem{lem}[thm]{Lemma}

\newtheorem*{cond*}{Condition}

\newtheorem*{lem*}{Lemma}
\theoremstyle{definition}

\newtheorem*{defn*}{Definition}

\newtheorem{fel*}[thm]{Exercise}

\newtheorem*{megf*}{Observation}
\theoremstyle{remark}
\newtheorem{rem}[thm]{Remark}

\newtheorem{obs}[thm]{Observation}
\newtheorem*{rem*}{Remark}

\title{On the growth rate of dichromatic numbers of finite subdigraphs}
\author{Attila Joó \thanks{University of Hamburg and
Alfréd Rényi Institute of Mathematics.
Funding was provided by the Alexander von Humboldt Foundation and partially by OTKA 129211.
 Email: {\tt attila.joo@uni-hamburg.de
 } }}
\date{2019}
\begin{document}
\maketitle
\begin{abstract}
Chris Lambie-Hanson proved recently that for every function $ f:\mathbb{N}\rightarrow \mathbb{N} $ there is an $ \aleph_1 
$-chromatic graph $ G $ of size $ 2^{\aleph_1} $ such that every $ (n+3) $-chromatic subgraph of $ G $ has at least $ f(n) $ 
vertices. Previously, this fact was just known to be consistently true due to P. Komjáth and S. Shelah. We 
investigate the  
analogue  
of this question for directed graphs. In the first part of the paper we give a simple method to construct for an arbitrary  $ 
f:\mathbb{N}\rightarrow 
\mathbb{N} $  an uncountably dichromatic digraph $ D $ of size $ 2^{\aleph_0} $  such that every $ (n+2) 
$-dichromatic subgraph of  $ D $ has at least $ f(n) $ vertices. In the second part we show that it is consistent with arbitrary 
large continuum that in the previous theorem ``uncountably dichromatic'' 
and ``of size $ 2^{\aleph_0} $'' 
can be  replaced  by ``$\kappa $-dichromatic'' and  ``of size $ \kappa $''  respectively where $ \kappa $ is universally 
quantified with bounds $ \aleph_0 \leq \kappa 
\leq 2^{\aleph_0}$. 
\end{abstract}
\section{Introduction}

The investigation of the finite subgraphs of uncountably chromatic graphs was initiated by Erdős and Hajnal  in the 1970s.  
First they were trying to construct uncountably chromatic graphs avoiding short cycles. After finding out that 
it is 
impossible they showed in \cite{Erdhaj} that every uncountably chromatic graph must contain every finite bipartite graph as a 
subgraph and exactly those are the ``obligatory'' finite subgraphs. An old conjecture of Erdős and Hajnal  has been recently 
justified by Chris Lambie-Hanson:

\begin{thm}[C. Lambie-Hanson \cite{Chris}]\label{Chris}
For every function $ f:\mathbb{N}\rightarrow \mathbb{N} $ there is an $ \aleph_1 
$-chromatic graph $ G $ of size $ 2^{\aleph_1} $ such that every $ (n+3) $-chromatic subgraph of $ G $ has at least $ f(n) $ 
vertices. 
\end{thm}

 P. Komjáth and S. Shelah 
showed earlier the consistency of the statement (guaranteeing 
an $ \aleph_1 $-sized $ G $) in 
\cite{KopeShelah}. They also proved that consistently, for every graph $ G $ of chromatic number at least $ \aleph_2 $ one 
can find 
graphs of arbitrary large chromatic number whose finite subgraphs are already induced subgraphs of  $ G $. If we replace $ 
\aleph_2 $ by $ \aleph_1 $ the resulting statement (known as Taylor conjecture) becomes consistently false.

\medspace
As a directed analogue of the chromatic number V. Neumann-Lara defined  the dichromatic number of a 
digraph $ D $ in \cite{dich} as the smallest cardinal $ \kappa $ such that $ V(D) $ can be partitioned into $ 
\kappa $ many sets each of spanning an acyclic subdigraph of $ D $. He and  Erdős 
conjectured that having  
chromatic number greater than  $ 
f(k) $ implies to have  orientation with dichromatic number  greater than $ k $  for a suitable $ f:\mathbb{N}\rightarrow 
\mathbb{N} $. This old conjecture is still wide open, even the  existence of $ f(3) $ is 
 unknown. Several results about dichromatic number are similar 
with the corresponding theorems about chromatic numbers. For example it was shown first by Bokal et al. in 
\cite{konstruktiv} 
using 
probabilistic methods and later by Severino in \cite{prob} in a constructive way that there are digraphs with arbitrary large 
finite 
dichromatic number avoiding directed cycles up to a given length.  Considering the infinite analogue of the question, D. T. 
Soukup showed that  in contrast to the
the behaviour of uncountably chromatic graphs, it is consistent that a digraph is uncountably dichromatic but avoids directed 
cycles up to  a prescribed length (see Theorem 3.5 of \cite{konz true}). Later it was shown that this is already true in ZFC 
(see 
\cite{sajat}), more 
precisely,  for 
every $ n \in \mathbb{N} $ and infinite cardinal $ \kappa $ there is a $ 
\kappa $-dichromatic digraph avoiding directed cycles of length up to $ n $. Even so, Soukup's forcing construction 
has additional strong properties and it has the flexibility to handle stronger statements. We will develop it further to prove 
Theorem \ref{I cons res}. Another result in \cite{konz true} says that the statement ``every graph of 
size and chromatic number $ \aleph_1 $ 
has an $ \aleph_1 
$-dichromatic orientation'' is 
independent of ZFC. It suggests that maybe the infinite version of the Erdős-Neumann Lara conjecture, where we consider 
arbitrary cardinals instead of natural numbers, is more approachable than the original.

Observe that 
avoiding short 
directed 
cycles can be 
formulated as a lower bound on the size of $ 2 $-dichromatic subgraphs. It seems natural to have such a bound for the $ n 
$-chromatic subgraphs for each $ n\in \mathbb{N} $ simultaneously.  The first result of the paper is the following directed 
analogue of Theorem~\ref{Chris}:
\begin{thm}\label{I ZFC result}
For every $ f: \mathbb{N} \rightarrow \mathbb{N} $, there is an uncountably dichromatic digraph  $ D $ of size $ 
2^{\aleph_0} $ such 
that   
every $ (n+2)$-dichromatic subdigraph of $ D $ has at least $ f(n) $ vertices.
\end{thm}
Under the continuum hypothesis the digraphs in Theorem \ref{I ZFC result} have ``optimal'' size (i.e. 
equal to their 
dichromatic number)  and settle the problem for $ \boldsymbol{\mathfrak{c}}:=2^{\aleph_0} $.   Our second result tells that 
it is 
consistent with arbitrarily large continuum that the analogue statement holds for every infinite $ \kappa\leq \mathfrak{c} $ 
with 
optimal sized 
digraphs. More precisely:
\begin{thm}\label{I cons res}
There is a ccc forcing $ \mathbb{P} $ of size $ \mathfrak{c} $ such that $ V^{\mathbb{P}}\models $ for every infinite 
cardinal  $ \kappa\leq \mathfrak{c} $ and for every $ f: 
\mathbb{N} \rightarrow \mathbb{N} $ there is a $ \kappa $-dichromatic digraph $ D $ on $ \kappa $ such that every $ (n+2 )
$-dichromatic subdigraph of $ D $ has at least $ f(n) $ vertices.
\end{thm}

Note that being ccc and having size $\mathfrak{c} $ ensures the preservation of all cardinals and it keeps the continuum the 
same thus  the forcing really accomplishes what we promised.

The paper is organized 
as follows. In the next section we introduce the necessary notation. The third section is subdivided into 
two parts in which we prove Theorems \ref{I ZFC result} and \ref{I cons res} respectively.

\section{Notation}
For an ordered pair $ \left\langle u,v \right\rangle  $ we write simply $ \boldsymbol{uv} $. The 
range of a function $ f $ is denoted by $ \boldsymbol{\mathsf{ran}(f)} $. The concatenation of sequences $ s $ and $ z $ 
is  $\boldsymbol{s ^{\frown}z} $ where sequences of length one are not distinguished in notation from their only 
elements. The Cartesian product of the sets $ X_i $ is $ \boldsymbol{\bigtimes_i X_i} $. The variable $ 
\boldsymbol{\kappa} $ 
is 
used for infinite cardinals, $ \boldsymbol{\alpha}, \boldsymbol{\beta}, \boldsymbol{\delta}  $ for 
ordinals and $ \boldsymbol{\omega} $ stands for the set of natural numbers. The set  subsets of $ X $ of size $ \kappa $ is 
denoted 
by $\boldsymbol{ [X]^{\kappa}} $ while $ \boldsymbol{ [X]^{<\kappa}} $ stands for the subsets smaller than $ \kappa $. 
About forcing we use the standard terminology and notation except that the ground model and the generic extension by 
generic filter $ \boldsymbol{G} $ are denoted by $ \boldsymbol{M} $ and $ \boldsymbol{M[G]} $ respectively instead of 
the more 
common $ V $ (which we preserve for vertex sets). 

A 
digraph  $ \boldsymbol{D} $ is a set 
of 
ordered  pairs 
without loops 
(i.e., without elements of the form $ vv $). 
A directed cycle $ C $ of size $ n \ (2\leq n<\omega) $ is a digraph of the form $ \{v_0v_1, v_1v_2,\dots, v_{n-1}v_0  \} $ 
where the $ v_i $ are pairwise distinct and  $ n\geq 2 $. The \textbf{digirth} of a  
$ D $ is 
the 
size of 
its smallest directed cycle if there are any, otherwise $ \infty $. A colouring of  the vertex set  $ \boldsymbol{V(D)}$ of $ D 
$  is \textbf{chromatic} (with respect to $ 
D $) if there is no monochromatic directed cycle. The \textbf{dichromatic number} $ \boldsymbol{\chi(D)} $ of $ D $ is the 
smallest 
cardinal $ \kappa $ such that $ D $ admits a chromatic colouring  with $ \kappa $ many colours. For $ U\subseteq V(D) $,  $ 
\boldsymbol{D[U]} $ denotes the subdigraph induced by $ U $ in $ D $.  Let $ 
\boldsymbol{f_D}(n):=\min \{ \left|U\right|: U\subseteq V(D) \wedge\chi(D[U])=n \} $ 
where $ \min 
\varnothing $ is considered $ \infty $. Note that if $ \chi(D)\geq \aleph_0 $ then by standard compactness arguments $ D $ has a finite $ n 
$-dichromatic subdigraph for every $ n<\omega $ and hence $ f_D $ has only finite values.

\section{Main results}
\subsection{The growing rate of  $\mathbf{f_D} $ for uncountably dichromatic $\mathbf{D} $ }
\begin{thm}\label{ZFC result}
For every $ f: \omega \rightarrow \omega $, there is an uncountably dichromatic digraph  $ D $ of size $ 2^{\aleph_0} $ such 
that   
every $ (n+2)$-dichromatic subdigraph of $ D $ has at least $ f(n) $ vertices.
\end{thm}

\begin{proof}
 To prove Theorem \ref{ZFC result}, it is enough to construct for every non-decreasing $ g:\omega \rightarrow \omega $  an 
 uncountably  
dichromatic digraph $ D $ of size $ 2^{\aleph_0} $
such that  whenever  $ H\subseteq D $ with $ \left|V(H)\right|<g(n) $ for some $  n <\omega $, we 
have $ 
\chi(H)\leq 2^{n} $.   Clearly, it is enough to consider only induced subdigraphs.

Let $ \boldsymbol{V}:=\bigtimes_{n<\omega} [0,g(n)-1] $ and for $ u\neq v\in V $, let $ uv\in \boldsymbol{D} $  if for the smallest $ m $ for 
which $ 
u(m)\neq v(m) $ we have $ v(m)=u(m)+1 $ mod $ g(m) $. For $ s\in \bigtimes_{k<n} [0,g(k)-1] $, let 
$ \boldsymbol{V_s}:=\{ v\in V: v \upharpoonright n=s\} $.
\begin{obs}\label{cycles are big}
For every $ s\in \bigtimes_{k<n} [0,g(k)-1] $,   $ D[V_s] $ has digirth $ g(n) $.
\end{obs}
\begin{lem}
 $ \chi(D)>\aleph_0 $
\end{lem}
\begin{proof}
Suppose for a contradiction that $ c: V(D) \rightarrow \omega $ is a chromatic colouring of $ D $. Colour~$ 0 $ cannot 
appear in all the sets $ V_i\ (i<g(0))  $ because otherwise by picking one 0-coloured vertex from each of those sets we would 
obtain a monochromatic 
directed cycle. We choose an $ i_0 $ such that colour $ 0 $ is not used in $ V_{i_0} $. Colour $ 1 $ cannot appear in all the 
sets $ V_{i_0^{\frown} i}\ (i<g(1))  $ because of similar reasons hence there is some $ i_1<g(1) $ such that colours $ 
0 $ and $ 
1 $ are not used in $ V_{i_0^{\frown} i_1} $.  By recursion we get a sequence $ v:=(i_n)_{n<\omega}\in V(D) $ such that 
none 
of the 
colours are used by $ c $ to colour vertex $ v $ which is a contradiction.
\end{proof}
\begin{lem}
$ f_D \geq f $
\end{lem}

\begin{proof}
Let  $ n<\omega $ be fixed and take a $ U\subseteq V $ of size less than $ g(n) $. We need to show  
$ {\chi(D[U])\leq 2^{n}}$. In the case $ n=0 $, Observation \ref{obs} says that $ D $ has digirth $ g(0) $, thus $ D[U] $ is 
acyclic.  For the case $ n >0 $, we define a chromatic colouring  $ c: U \rightarrow \{ 0,1 
\}^{n} $ by setting    $ {\boldsymbol{c(u)}:=(\mathsf{sgn}(u(0)),\mathsf{sgn}(u(1)),\dots, 
\mathsf{sgn}(u(n-1))) }$  for  $ u\in U $
(here $ \mathsf{sgn}(0)=0 $ and $ 
\mathsf{sgn}(n)=1 $ for $n>0 $).  
To prove that $ c $ is a chromatic colouring, suppose that  $ C $ is a directed cycle in $ D[U] $ and let $ s $ be the longest 
common 
initial segment of the vertices in $ 
 V(C) $. Since $ 
\left|V(C)\right|\leq \left|U\right|<g(n) $, Observation \ref{cycles are big} guarantees that  $n > \left|s\right|=: 
\boldsymbol{m}$. From the 
structure of $ D $ is 
clear that we must have $ V(C)\cap V_{{s}^{\frown} 
i}\neq \varnothing $ for every $ i<g(m) $. Then for  a $ u_0\in V(C)\cap V_{{s}^{\frown} 
0} $ and $ u_1\in V(C)\cap V_{{s}^{\frown} 
1} $, $ c(u_0)\neq c(u_1) $ (because $ c(u_0)(m)=0\neq 1= c(u_1)(m)$). Thus $ c $ is a chromatic colouring  and hence $ 
\chi(D[U])\leq 2^{n}$.
\end{proof}
\end{proof}

\begin{rem}\label{omega OK}
For every $ f: \omega\rightarrow \omega $ there is an $ \aleph_0 $-dichromatic digraph $ D $ on $ \omega $ such that $ f_D 
\geq f $. Indeed, it follows from Theorem \ref{ZFC result}  via compactness arguments that for a fixed $ f$  for every $ 
n<\omega $ there is a 
finite $(n+2) 
$-dichromatic digraph $ D_n $ for which $ f_{D_n}\geq f $. By taking disjoint copies of the  digraphs $ 
D_n $  we obtain a 
desired $ D $.
\end{rem}

\subsection{A consistency result about  $ \leq 2^{\aleph_0} $-dichromatic digraphs}

We restate the result here in a slightly stronger form. Let us remind that in light of Remark \ref{omega OK} we do not have 
to 
bother with the   case $ 
\kappa=\aleph_0 $.
\begin{thm}\label{cons res}
There is a ccc forcing $ \mathbb{P} $ of size $ \mathfrak{c} $ such that $ \Vdash_\mathbb{P} $ for every uncountable 
cardinal  $ \kappa\leq \mathfrak{c} $ and for every $ f: 
\omega \rightarrow \omega $, there is a  digraph $ D $ on $ \kappa $ with $ f_D \geq f $ such that for every  uncountable $ 
U\subseteq V(D) $ we have $ \chi(D[U])=\left|U\right| $, in particular $ \chi(D)=\kappa $. 
\end{thm}

\begin{proof}
Without loss of generality we can assume that $ f $ is non-decreasing.  We start with some 
basic estimations that we need later. If $ D, H $ are digraphs then a 
function $ 
\varphi: V(D)\rightarrow V(H) $ 
is a  \textbf{semihomomorphism} from $D $ 
to $ H $ if for each $ 
uv\in D $
either $ \varphi(u)\varphi(v)\in H $ or $ \varphi(u)=\varphi(v) $. A semihomomorphism is \textbf{acyclic} if for every $ v\in 
\mathsf{ran}(\varphi) $, $ \varphi^{-1}(v) $ 
spans an acyclic  
subdigraph of $ D $ (shortly $ \chi(D[ \varphi^{-1}(v)])=1 $).

\begin{prop}\label{acyc hom0}
If $ D, H $ are digraphs and $ \varphi $ is an acyclic semihomomorphism from $D $ to $ H $, then $ \chi(D)\leq \chi(H) $. 
\end{prop}
\begin{proof}
If $ c$ is a chromatic colouring of $ H $ then so is $   c\circ \varphi $ for $ D $.
\end{proof}

\begin{prop}\label{acyc hom}
Let $ D, H $ be digraphs and assume that there is a   semihomomorphism $ \varphi $ from $D $ to $ H $ where  $ 
\chi(D[\varphi^{-1}(v)])=k_v+1 
$ for $ v\in \mathsf{ran}(\varphi)$. Then $ \chi(D)\leq \chi(H)+\sum_{v\in \mathsf{ran}(\varphi)}k_v
$. 
\end{prop}
\begin{proof}
For $ v\in \mathsf{ran}(\varphi)$, fix a  chromatic colouring $ c_v $ of $ D[\varphi^{-1}(v)] $ with the colours 
$ \{0,\dots,k_v  \}  $.  For every $ u\in V(D) $ with  $ c_{\varphi(u)}(u)\neq 0 $,  colour $ u $ with  the ordered pair $ 
\left\langle 
\varphi(u),c_{\varphi(u)}(u)
\right\rangle  $. Delete all the $ u $ we already coloured. Since 
from  each $ D[\varphi^{-1}(v)] $ all but the vertices $ u $ with $ c_v(u) = 0 $ have been deleted, the restriction of $ \varphi 
$ to 
the remaining digraph is an acyclic semihomomorphism. By Proposition \ref{acyc hom}, it has a chromatic 
colouring with the colours, say $ \{ 0,\dots,\chi(H)-1 \} $. We defined a chromatic colouring of $ D $ with $ 
\chi(H)+\sum_{v\in \mathsf{ran}(\varphi)}k_v $ colours witnessing the desired inequality.
\end{proof}

To continue the proof of Theorem \ref{cons res},  let $ \kappa $ be an uncountable cardinal  and  let $ f:\omega\rightarrow 
\omega 
$. We define $ 
\pmb{\mathbb{P}_{\kappa,f}} $ to be  the poset where $ p\in \mathbb{P}_{\kappa,f} $ if $ p $ is a digraph with $ 
V(p)\in 
[\kappa]^{<\aleph_0} 
$  satisfying $ f_p \geq f $ and  $ q\leq p $ if $ p\subseteq q$. 

\begin{lem}\label{ccc}
$ \mathbb{P}_{\kappa,f} $   satisfies ccc. 
\end{lem}
\begin{proof}
Let $ \{ p_\alpha: \alpha<\omega_1 \}\subseteq \mathbb{P}_{\kappa,f} $. By the $ \Delta $-system lemma, there is an 
uncountable $ \boldsymbol{U}\subseteq \omega_1 $ such that $ \{ V(p_\alpha): \alpha \in U \} $ forms a $ \Delta $-system 
with root $ \boldsymbol{R} $.  We consider for $ \alpha\in U $  the 
following first order 
structure $ 
\boldsymbol{\mathcal{A}_\alpha} $  on ground set $ V(q_\alpha) $: we have a binary relation defined by the digraph  $ 
p_\alpha $, a linear order $ \in $ given by the fact that the
elements of $ V(q_\alpha) $ are ordinals and constants for each element of $ R $. Up to isomorphism there 
are 
just 
finitely 
many such a first order structures on a finite ground set  therefore there is an uncountable $ \boldsymbol{U'}\subseteq 
U $ such that for $ \alpha\in U' $, the $ 
\mathcal{A}_\alpha $ are pairwise isomorphic. Note that the isomorphism between two $ \mathcal{A}_\alpha $ is uniquely 
determined by the linear order, 
and its restriction to $ R $ is the identity. 

We show that for  $ \beta\neq\delta\in U' $, $ p_\beta\cup p_\delta\in  
\mathbb{P}_{\kappa,f} $. Let $ V(p_\beta)=\{ \beta_0,\dots, \beta_{n-1} \} $ and $ V(p_\delta)=\{ \delta_0,\dots, 
\delta_{n-1} \} $ where the enumerations are in increasing order. Consider $ \boldsymbol{\varphi}:  p_\beta\cup p_\delta
\rightarrow p_\delta$ where $
\varphi(\beta_i):=\delta_i $  and $ \varphi(\delta_i):=\delta_i $. Since $ \mathcal{A}_\beta$ and $ \mathcal{A}_\delta $ are 
isomorphic, $ \varphi $ is a semihomomorphism.
 The inverse image of a vertex $ \delta_i $ is $ \{ 
\beta_i, \delta_i \} $ which is a singleton if  $ \delta_i \in R $ and a vertex pair without any edge between them 
 otherwise.  Hence  $ \varphi$ is an acyclic semihomomorphism. Let  $ 
 \boldsymbol{W}\subseteq V(p_\beta\cup p_\delta) $ be 
 arbitrary and $ 
 \boldsymbol{W^{*}}:=\varphi[W] $. We write $ \boldsymbol{k} $ and $\boldsymbol{ k^{*}} $ for $ \chi( (p_\beta\cup 
 p_\delta)[W]) $ and $ 
 \chi( (p_\beta\cup p_\delta)[W^{*}]) $ 
 respectively. Then $ \left|W^{*}\right|\geq f(k^{*}) $ because $ p_\delta\in \mathbb{P}_{\kappa,f} $.   Proposition 
 \ref{acyc hom0} guarantees $ k^{*}\geq k $ from which $ f(k^{*})\geq f(k) $ follows since $ f $ is assumed to be 
 non-decreasing. By combining these facts, we obtain
 
 \[  \left|W\right|\geq \left|W^{*}\right|\geq f(k^{*})\geq f(k).\]
 
Since $ W $ was arbitrary, we may conclude that $ f_{p_\beta\cup p_\delta}\geq f $ and hence  $ p_\beta\cup p_\delta\in 
\mathbb{P}_{\kappa,f}  $.
\end{proof}

Suppose that $ G $ is a $ \mathbb{P}_{\kappa,f} $-generic 
filter and let us define $ \boldsymbol{D}:=\bigcup G $.

\begin{obs}\label{obs}
 $ D $  is a digraph on $ \kappa $ satisfying  $ f_D \geq f $.
\end{obs}

\begin{lem}\label{Dani cucc}
$ \chi(D[U])=\left|U\right| $ holds for every  uncountable $ 
U\subseteq V(D) $. 
\end{lem}
\begin{proof}

It is enough to show that for every $ U\in [\kappa]^{\aleph_1} $ the digraph $ D[U] $ contains some directed cycle. Let  
$U\subseteq V(D) $ be uncountable forced by the condition $ \boldsymbol{p}\in G $. It is enough to show that $ 
\boldsymbol{S}:=\{ 
q\in 
\mathbb{P}_{\kappa,f}: q\Vdash  \chi(\dot{D}[\dot{U}])\geq 2\} $ is dense below  $ p $ (note that $ \chi(D[U])\geq 2 $ 
means that $ U $ spans some 
directed cycle in $ D $). Let $ \boldsymbol{r}\leq p $ be given. For every $ \alpha\in U $, we pick an $ 
\boldsymbol{q_\alpha }\leq r $ 
such that  $ q_\alpha \Vdash  \check{\alpha}\in \dot{U} $ and  $\alpha\in  
V(q_\alpha) $. We proceed similarly as in the proof of Lemma \ref{ccc}. By applying the $ \Delta $-system lemma, we trim 
$ U $ to an 
uncountable 
$\boldsymbol{ U'}\subseteq U $  where $ \{ V(p_\alpha): \alpha \in U \} $ form a $ \Delta $-system with root $ 
\boldsymbol{R} $. For $ \alpha\in U' $,  let $ \boldsymbol{\mathcal{B}_\alpha} $ be the first order structure  on ground set 
$ V(q_\alpha) $ where we have a constant that stands for $ \alpha $ in $ \mathcal{B}_\alpha $, a binary 
relation defined by the digraph  $ 
q_\alpha $, a linear order $ \in $ given by the fact that the
elements of $ V(q_\alpha) $ are ordinals and one constant for each element of $ R $. Up to isomorphism there are just 
finitely 
many such a first order structures on a finite ground set therefore there is an uncountable $ 
\boldsymbol{U''}\subseteq 
U' $ such that for $ \alpha\in U'' $, the $ 
\mathcal{B}_\alpha $ are pairwise isomorphic. Let $ \boldsymbol{n} $ be the common size of the ground sets of the 
structures $ 
\{ \mathcal{B}_\alpha: \alpha\in U'' \} $. Note that $  U''\cap R =\varnothing $ otherwise  we would have $ 
U''\subseteq R $ contradicting the fact that $ R $ is finite. We pick  $ \boldsymbol{m}:=n+\max_{k\leq n}f(k+1)-f(k) $ many
elements, say $ 
\boldsymbol{\alpha_0},\dots,\boldsymbol{\alpha_{m-1}}$, of $ U'' $ and define the directed cycle $ \boldsymbol{C}:=\{ 
\alpha_0\alpha_1, 
\alpha_1\alpha_2,\dots 
,\alpha_{m-1}\alpha_0 \} $. To simplify the notation, from now on we write simply $ \boldsymbol{q_i} $ instead of $ 
q_{\alpha_i} $ and let us 
define  
$\boldsymbol{q}:=C\cup\bigcup_{i<m}q_{i}  $.
\begin{claim}
 $  q\in  \mathbb{P}_{\kappa,f} $.
\end{claim}
\begin{proof}
The only nontrivial part of the claim is that $ f_q\geq f $ holds. 
Let 
$ V(q_i)=\{ \beta_{i,j}: j<n \} $ where the enumeration is in increasing order. There is a 
$ j_0<n $ such that $ \alpha_i=\beta_{i,j_0} $ for every $ i<m $ because the $ \mathcal{B}_{\alpha_i} $ are pairwise 
isomorphic. Consider the function $ \varphi: V(q)\rightarrow V(q_{0}) $ where $  \varphi(\beta_{i,j})= \beta_{0,j}  
$. The inverse image  of a $ v\in V(q_{0}) $ with respect to  $ \varphi $ is: $ \{ v \} $ if $ v\in 
R $, the directed 
cycle $ C $ if $ v=\alpha_0 $ and an independent set of size $ m $ otherwise. Combining this with the fact that $ 
\mathcal{B}_{\alpha_i} $ are 
pairwise 
isomorphic, we may conclude that $ \varphi $ is a semihomomorphism from $ q $ to $ q_0 $. Let  $ 
\boldsymbol{W}\subseteq V(q) $ be 
arbitrary and $ 
\boldsymbol{W^{*}}:=\varphi[W] $. We write $ \boldsymbol{k} $ and $\boldsymbol{ k^{*}} $ for $ \chi(q[W]) $ and $ \chi(q[W^{*}]) $ 
respectively.   If
$ V(C)\not\subseteq W $, then   $ \varphi \upharpoonright W $  is an acyclic semihomomorphism from $ q[W] $ to
$ q_{0}[W^{*}] $ and hence $ k\leq k^{*} $. By combining this with the facts that $ f_{q_{0}}\geq f$ and $ f $ is 
non-decreasing,  we obtain
\[  \left|W\right|\geq \left|W^{*}\right|\geq f(k^{*})\geq f(k).\]
If $ V(C)\subseteq W $ then $ \left|W\right|\geq \left|V(C)\right|=m $ and hence $ \left|W \right|-\left|W^{*}\right|\geq m-n $. 
In 
this 
case  we  apply Proposition \ref{acyc hom} with $ \varphi \upharpoonright 
W $. Using the terminology of the Proposition, $ k_{v}=1 $ if $ v=\alpha_0 $ and $ k_{v}=0 $ otherwise
thus $ k\leq k^{*}+1 $. By the choice of $ m $ we obtain
\[  \left|W\right|\geq \left|W^{*}\right|+(m-n)\geq f(k^{*})+(m-n)\geq f(k^{*}+1)\geq f(k).\]
\end{proof}
 Clearly $ q\leq q_i $ for $ i<m $ and therefore  $ q \Vdash \bigwedge_{i<m}\check{\alpha}_i\in  \dot{U}$.  Because of $ 
 C\subseteq q $ 
we also have 
 $q\Vdash \chi(\dot{D}[\dot{U}] )\geq 2$. Since $ r\leq p $ was arbitrary and $ q\leq q_i\leq r $, we may conclude that $ S $ is dense below  $ p $.
\end{proof}
We build the $ \mathbb{P} $ in Theorem \ref{cons res} as the $ \mathbb{P}_{\mathfrak{c}} $ of a finite support iteration $ 
(\mathbb{P}_\alpha, \dot{\mathbb{Q}}_\beta)_{\alpha\leq \mathfrak{c}, \beta <\mathfrak{c}} $   of ccc posets of 
size at most $ 
\mathfrak{c} $ which ensures that $ \mathbb{P} $ is ccc and $ \left|\mathbb{P}\right|=
\mathfrak{c} $. We let every (non-trivial)  factor $ \dot{\mathbb{Q}}_\beta $ to  be 
$\dot{\mathbb{P}}_{\check{\kappa}, \dot{f}}  $ for some
 $ \aleph_0<\kappa\leq 2^{\aleph_0} $ and for a  nice  $ \mathbb{P}_\delta $-name $ \dot{f} $ of a function $ f:\omega 
\rightarrow \omega $ where $ \delta<\beta $.  Lemma \ref{ccc} ensures that the factors are really ccc. By standard 
bookkeeping techniques,  the iteration
can be 
organized in the 
way that for every $ \mathbb{P} $-generic filter $ G $, uncountable  $ \kappa\leq \mathfrak{c} $ and  $ f:\omega 
\rightarrow \omega $ living in $ M[G] $, there is a 
factor  $ \dot{\mathbb{Q}}_\beta =\dot{\mathbb{P}}_{\check{\kappa}, \dot{f}} $ in the iteration. Then by Observation 
\ref{obs}, the digraph $ D $ given by the intermediate forcing   $\mathbb{Q}_\beta=\mathbb{P}_{\kappa,f} $ has size $ 
\kappa $ and  satisfies $ f_D 
\geq f $. To justify $ M[G]\models \forall U\in 
 [\kappa]^{\aleph_1}: \chi(D[U])\geq 2  $,  consider the forcing $ \mathbb{P}_{\geq \beta} $  over the intermediate extension 
 $ M[G_{<\beta}] $). From this point the proof  goes the same way as the proof of
 Lemma \ref{Dani cucc} by  working formally with the whole 
 $ \mathbb{P}_{\geq \beta} $ instead of just $\mathbb{Q}_\beta=\mathbb{P}_{\kappa,f} $. More precisely, whenever we  
 deal with a condition $ p $ in the original proof, we consider now just its initial coordinate $ p(\beta) $.
 
\end{proof}


\begin{bibdiv}
\begin{biblist}

\bib{Erdhaj}{article}{
   author={Erd\H{o}s, P.},
   author={Hajnal, A.},
   title={On chromatic number of graphs and set-systems},
   journal={Acta Math. Acad. Sci. Hungar.},
   volume={17},
   date={1966},
   pages={61--99},
   issn={0001-5954},
   review={\MR{193025}},
   doi={10.1007/BF02020444},
}
		
\bib{Chris}{article}{
      title={On the growth rate of chromatic numbers of finite subgraphs},
        author={Lambie-Hanson, Chris},
        review={ \url{https://arxiv.org/abs/1902.08177}},
        year={2019}
        }
        
\bib{KopeShelah}{article}{
   author={Komj\'{a}th, P\'{e}ter},
   author={Shelah, Saharon},
   title={Finite subgraphs of uncountably chromatic graphs},
   journal={J. Graph Theory},
   volume={49},
   date={2005},
   number={1},
   pages={28--38},
   issn={0364-9024},
   review={\MR{2130468}},
   doi={10.1002/jgt.20060},
}

\bib{konz true}{article}{
   author={Soukup, D\'{a}niel T.},
   title={Orientations of graphs with uncountable chromatic number},
   journal={J. Graph Theory},
   volume={88},
   date={2018},
   number={4},
   pages={606--630},
   issn={0364-9024},
   review={\MR{3818601}},
   doi={10.1002/jgt.22233},
}
\bib{sajat}{article}{
      title={Uncountable dichromatic number without short directed cycles},
        author={Joó, Attila},
        review={\url{https://arxiv.org/abs/1905.00782}},
        year={2019}
        }

\bib{chrom survey}{article}{
   author={Komj\'{a}th, P\'{e}ter},
   title={The chromatic number of infinite graphs---a survey},
   journal={Discrete Math.},
   volume={311},
   date={2011},
   number={15},
   pages={1448--1450},
   issn={0012-365X},
   review={\MR{2800970}},
   doi={10.1016/j.disc.2010.11.004},
}
	

\bib{dich}{article}{
   author={Neumann Lara, V.},
   title={The dichromatic number of a digraph},
   journal={J. Combin. Theory Ser. B},
   volume={33},
   date={1982},
   number={3},
   pages={265--270},
   issn={0095-8956},
   review={\MR{693366}},
   doi={10.1016/0095-8956(82)90046-6},
}

\iffalse
\bib{extrem}{article}{
   author={Erd\H{o}s, Paul},
   author={Gimbel, John},
   author={Kratsch, Dieter},
   title={Some extremal results in cochromatic and dichromatic theory},
   journal={J. Graph Theory},
   volume={15},
   date={1991},
   number={6},
   pages={579--585},
   issn={0364-9024},
   review={\MR{1133813}},
   doi={10.1002/jgt.3190150604},
}
\fi
\bib{prob}{article}{
   author={Bokal, Drago},
   author={Fijav\v{z}, Ga\v{s}per},
   author={Juvan, Martin},
   author={Kayll, P. Mark},
   author={Mohar, Bojan},
   title={The circular chromatic number of a digraph},
   journal={J. Graph Theory},
   volume={46},
   date={2004},
   number={3},
   pages={227--240},
   issn={0364-9024},
   review={\MR{2063373}},
   doi={10.1002/jgt.20003},
}
\bib{konstruktiv}{article}{
   author={Severino, Michael},
   title={A short construction of highly chromatic digraphs without short
   cycles},
   journal={Contrib. Discrete Math.},
   volume={9},
   date={2014},
   number={2},
   pages={91--94},
   issn={1715-0868},
   review={\MR{3320450}},
}

\end{biblist}
\end{bibdiv}
\end{document}